\documentclass[english]{article}
\usepackage[latin1]{inputenc}
\usepackage{babel}
\usepackage{amsmath,amsfonts,amssymb,amsthm}
\usepackage{bbm}
\usepackage{hyperref}
\usepackage{verbatim}
\usepackage{tikz}
\usepackage{tikz-cd}
\usetikzlibrary{patterns}
\usepackage{pstricks}

\theoremstyle{plain}
\newtheorem{teo}{Theorem}[section]
\newtheorem{cor}[teo]{Corollary}
\newtheorem{prop}[teo]{Proposition}
\newtheorem{lem}[teo]{Lemma}
\newtheorem{claim}[teo]{Claim}
\newtheorem{rem}[teo]{Remark}
\newtheorem{defin}[teo]{Definition}
\newtheorem{Q}[teo]{Question}

\newcommand{\system}[1]{\mbox{\fontfamily{cmss}\fontshape{n}\fontseries{m}\selectfont#1}}
\newcommand{\ZF}{\system{ZF}}
\newcommand{\ZFC}{\system{ZFC}}
\newcommand{\AC}{\system{AC}}

\newcommand{\AD}{\system{AD}}

\DeclareMathOperator{\cof}{cof}
\DeclareMathOperator{\crt}{crt}
\DeclareMathOperator{\OD}{OD}
\DeclareMathOperator{\HOD}{HOD}

\DeclareMathOperator{\Ult}{Ult}

\title{LD-algebras beyond I0}
\author{Vincenzo Dimonte\footnote{Universit\`{a} degli Studi di Udine, via delle Scienze, 206 33100 Udine (UD) \emph{E-mail address:} \texttt{vincenzo.dimonte@gmail.com}}}

\begin{document}

\maketitle

\begin{abstract}

The algebra of embeddings at the I3 level has been deeply analyzed, but nothing is known algebra-wise for embeddings above I3. In this paper is introduced an operation for embeddings at the level of I0 and above, and it is proven that they generate an LD-algebra that can be quite different from the I3 one.



\end{abstract}

\section{Introduction}

The connection between large cardinals and LD-algebras is one of the most intriguing success stories of the theory of large cardinals. LD-algebras are algebras with one operator that satisfies the left-distributive law, i.e., 
\begin{equation*}
 \forall x,y,z\ x*(y*z)=(x*z)*(x*y). 
\end{equation*}
At first sight, they have nothing to do with large cardinals, as they can be small, countable, even finite. Large cardinals above a certain point, on the other hand, are always defined by elementary embeddings. At the top of the large cardinal hierarchy there are the so-called rank-into-rank embeddings: the weakest ones are called I3 (i.e., the existence of $j:V_{\lambda}\prec V_\lambda$), then I2 is stronger then I3, I1 is stronger than I2, and so on. These hypotheses are exorbitantly strong, stronger than any large cardinal ``normally'' used (for example, under I3$(\lambda)$, $\lambda$ is limit of cardinals that are $n$-huge for any $n\in\omega$). Yet, it is possible to define an operation on the embeddings for I3 that is left distributive. Laver \cite{Laver1} proved that the algebra generated by one embedding is free, and therefore isomorphic to $F_1$, the free LD-algebra with one generator.

The beauty in this approach is that now we can use all the strength and peculiarities of elementary embeddings to prove results on the algebra of embeddings, and then all these results will be automatically transfered to $F_1$, that is a countable "simple" object, living in \ZFC. For example, this approach was used to prove that in $F_1$ the word problem is decidable, and that left division is a linear ordering. With time, the same things were proved under \ZFC, but I3 pointed the way, and there are still some results for which we do not know whether I3 is necessary. For more information on this, \cite{Dehornoy1} is an exhaustive survey, while \cite{Dehornoy2} explores in depth the algebraic part.

It is natural to ask if the same trick can be used for more generators. It is still open:
\begin{Q} 
 (I3) Are there $j,k:V_\lambda\prec V_\lambda$ such that the LD-algebra generated by them is free?
\end{Q}

One way to approach this problem is to look for stronger hypotheses. Up to I1, actually, the structure of the algebra generated by two I1-embeddings is isomorphic to the one generated by I3-embeddings. On I0, the definition of the operator as in I3 does not work, and one should find a new definition (even if it works only on embeddings with a certain property). Yet, the structure again is not new. To find something really different we have to climb up the hierarchy above I0.

In this paper, we introduce an operator on embeddings that witness hypotheses above I0, that is still consistent with the operation on I3 and that generates an LD-algebra. As the theory of the hypotheses above I0 is varied and with plenty of different situations, this will provide an abundance of new LD-algebras to work with. As an example, two embeddings are introduced that enjoy strong independence properties (even if it is still not clear whether they produce a free algebra).

The definition of the operator uses key properties of the $E^0_\alpha$-hierarchy, so much of the preliminaries is dedicated to its introduction and definition. In the rest of the paper the application is defined on a particular kind of elementary embeddings, whose existence is derived from the $E^0_\alpha$-hierarchy, it is proven that the application generates an LD-algebra and it is analyzed how much such an algebra is similar or different from the I3 case.

\section{Preliminaries}

 To avoid confusion or misunderstandings, all notations and standard basic results are collected here.

  The double arrow (e.g. $f:a\twoheadrightarrow b$) denotes a surjection.
  
	If $X$ is a set, then $L(X)$ denotes the smallest inner model of \ZF{} that contains $X$; it is defined like $L$ but starting with the transitive closure of $\{X\}$ as $L_0(X)$.

  If $X$ is a set, then $\OD_X$ denotes the class of the sets that are \emph{ordinal-definable over $X$}, i.e., the sets that are definable using ordinals, $X$ and elements of $X$ as parameters. $\HOD_X$ denotes the class of the sets that are \emph{hereditarily ordinal-definable over $X$}, i.e., the sets in $\OD_X$ such that all the elements of their transitive closure are in $\OD_X$.  For example, $L(X)\vDash V=\HOD_X$. One advantage in considering models of $\HOD_X$ is the possibility of defining partial Skolem functions. Let $\varphi(v_0, v_1, \dots,v_n)$  be a formula with $n+1$ free variables and let $a\in X$. Then:
  \begin{equation*}
    h_{\varphi,a}(x_1,\dots,x_n)=\begin{cases}
      y                   & \text{where }y\text{ is the least in }OD_{\{a\}}\text{ such that }\\
                       & \quad\varphi(y,x_1,\dots,x_n)\\
      \emptyset           & \text{if }\forall x\neg\varphi(x,x_1,\dots,x_n)\\
      \text{not defined} & \text{otherwise}
    \end{cases}
  \end{equation*}
  are partial Skolem functions. For every set or class $y$, $H^{L(X)}(y)$ denotes the closure of $y$ under partial Skolem functions for $L(X)$, and $H^{L(X)}(y)\prec L(X)$.
	
  If $M$ and $N$ are sets or classes, $j:M\prec N$ denotes that $j$ is an elementary embedding from $M$ to $N$, that is an injective function such that for any formula $\varphi$ and any $x\in M$, $M\vDash\varphi(x)$ iff $N\vDash\varphi(j(x))$. The case in which $j$ is the identity, i.e., if $M$ is an elementary submodel of $N$, is simply written as $M\prec N$.

  If $M\vDash\AC$ or $N\subseteq M$ and $j:M\prec N$ is not the identity, then it moves at least one ordinal. The \emph{critical point} of $j$, $\crt(j)$, is the least ordinal moved by $j$.

  Let $j$ be an elementary embedding and $\kappa=\crt(j)$. Define $\kappa_0=\kappa$ and $\kappa_{n+1}=j(\kappa_n)$. Then $\langle \kappa_n:n\in\omega\rangle$ is the \emph{critical sequence} of $j$.
  
  Kunen \cite{Kunen} proved that if $M=N=V_\eta$ for some ordinal $\eta$, and $\lambda$ is the supremum of the critical sequence, then $\eta$ cannot be bigger than $\lambda+1$ (and of course cannot be smaller than $\lambda$).
	
Kunen's result actually does not say anything about the cases $\eta=\lambda$ or $\eta=\lambda+1$. Therefore we can introduce the following hypotheses without fearing an immediate inconsistency:

\begin{description}
 \item[I3] There exists $j:V_\lambda\prec V_\lambda$, where $\lambda$ is the supremum of the critical sequence of $j$.
 \item[I0] There exists $j:L(V_{\lambda+1})\prec L(V_{\lambda+1})$, where $\crt(j)<\lambda$.
\end{description}

We add $\crt(j)<\lambda$ so that I0 implies I3 (see \cite{Kanamori} for the definitions of I2 and I1). It is immediate to see that if I0 holds for $j$, then $\lambda$ is the supremum of the critical sequence of $j$.

The most fruitful consequences of I3 are its connections with algebra. Call ${\cal E}_\lambda=\{j:j:V_\lambda\prec V_\lambda\}$. Then we can define an operation on ${\cal E}_\lambda$, the application: if $j,k\in{\cal E}_\lambda$, $j\cdot k=\bigcup_{n\in\omega} j(k\cap V_{\kappa_n})$, where $\kappa_n$ is the critical sequence of $j$. Alternative notations that will be used are $j(k)$ and $j^+(k)$. Note that application should not be confused with the more common composition, that is very different: for example $\crt(j(j))=j(\crt(j))>\crt(j)$, but $\crt(j\circ j)=\crt(j)$. 

\begin{prop} [Laver \cite{Laver2}]
 $j\cdot k\in{\cal E}_\lambda$.
\end{prop}

 Moreover, $j\cdot (k\cdot h)=(j\cdot k)\cdot(j\cdot h)$. This is called the left distributive law, and $({\cal E}_\lambda, \cdot)$ is a left distributive algebra, or LD-algebra.

One can see the LD-algebra in a more abstract way. Let $T_n$ be the set of words constructed using the variables $x_1,\dots,x_n$ and a binary operator $\cdot$. Denote $\equiv_{LD}$ the congruence on $T_n$ generated by all pairs of the form $t_1\cdot(t_2\cdot t_3)$, $(t_1\cdot t_2)\cdot(t_1\cdot t_3)$. Then $T_n/\equiv_{LD}$ is a free LD-algebra with $n$ generators. We call it $F_n$.

Given an LD-algebra, we can consider its subalgebra ${\cal A}_X$ generated by the elements in a finite subset $X$. We say that the subalgebra is free iff it is isomorphic to $F_{|X|}$. By the universal property of $F_{|X|}$, there is always a surjective homomorphism $\pi$ from $F_{|X|}$ to ${\cal A}_X$ (it sends the generators $x_1,\dots, x_n$ to $X$ and is a morphism for the operator), therefore ${\cal A}_X$ is free iff such homomorphism is also injective iff for any two $t_1,t_2\in T_{|X|}$, $\pi(t_1)=\pi(t_2)$ iff $t_1\equiv_{LD} t_2$.

Let ${\cal A}_j$ be the subalgebra of ${\cal E}_\lambda$ generated by $\{j\}$. 

\begin{teo}[Laver \cite{Laver1}]
 \label{free}
 ${\cal A}_j$ is a free LD-algebra.
\end{teo}

I0 has instead received attention because of its similarities with $\AD^{L(\mathbb{R})}$. Woodin in \cite{Woodin} tried to push these similarities even further, creating a hierarchy of new hypotheses stronger than I0, with the objective of finding a hypothesis similar to $\AD_\mathbb{R}$. To do this, instead of dealing with $L(V_{\lambda+1})$, we deal with $L(N)$, where $V_{\lambda+1}\subseteq N\subseteq V_{\lambda+2}$ and $N=L(N)\cap V_{\lambda+2}$, and with embeddings from $L(N)$ to itself. Actually, we are interested only in models of the type $L(X,V_{\lambda+1})$ with $X\subseteq V_{\lambda+1}$, but it turns out that it is advantageous to start working with more generality.

We will work in $L(N)$'s that do not satisfy the Axiom of Choice. Like in $L(\mathbb{R})$, it is possible to define a cardinal in $L(N)$ that ``measures'' the largeness of $V_{\lambda+1}$:
  \begin{defin}
    Let $M$ be a set or a class such that $V_{\lambda+1}\subseteq M$. Then $\Theta^M$ is the supremum of the ordinals $\alpha$ such that there exists $\pi:V_{\lambda+1}\twoheadrightarrow\alpha$ with $\{(a,b)\in V_{\lambda+1}\times V_{\lambda+1}:\pi(a)<\pi(b)\}\in M$. If $M$ is a class, then this is equivalent to the more classical definition:
    \begin{equation*}
     \Theta^M=\sup\{\alpha:\exists\pi:V_{\lambda+1}\twoheadrightarrow\alpha,\ \pi\in M\}.
    \end{equation*}
  \end{defin}
	
  Note that $\Theta^{L(N)}$ is a cardinal in $L(N)$, and $\lambda^+<\Theta^{L(N)}\leq(2^\lambda)^+$. Moreover, if $L(N)\cap V_{\lambda+2}=N$ then $\Theta^{L(N)}=\Theta^N$. 
	
	Unlike the I3-embeddings, embeddings from $L(N)$ to itself have a nice property: we can assume without loss of generality that they are generated by an ultrafilter:
	
	\begin{teo}[\cite{Woodin}]
    \label{LosN}
    Let $V_{\lambda+1}\subseteq N\subset V_{\lambda+2}$ be such that $L(N)\cap V_{\lambda+2}=N$ and let $j:L(N)\prec L(N)$. Then there exists an ultrafilter $U\subset N$ such that $\Ult(L(N),\ U)$ is well-founded. By condensation the collapse of $\Ult(L(N),\ U)$ is $L(N)$ and $j_U:L(N)\prec L(N)$, the inverse of the collapse, is an elementary embedding with $\crt(j)<\lambda$. Moreover, there is an elementary embedding $k_U:L(N)\prec L(N)$ that is the identity on $N$ and such that $j=j_U\circ k_U$.
  \end{teo}
	
	 By Theorem \ref{LosN} any elementary embedding $j:L(N)\prec L(N)$ can be factored into two elementary embeddings, $j=j_U\circ k$. The first embedding, $j_U$, is obtained from an ultrafilter, and it is completely determined by its behaviour on $N$; the second one, $k$, is the identity on $N$ and moves only larger cardinals, and hence can be generated by a shift of indiscernibles. Note that $j_U$ witnesses I0, while $k$ seems only a combinatorial permutation.

  \begin{defin}
    \label{proper}
    Let $V_{\lambda+1}\subseteq N\subset V_{\lambda+2}$ be such that $L(N)\cap V_{\lambda+2}=N$ and let $j:L(N)\prec L(N)$. For every $a\in L(N)$, we will indicate with $\langle a_0,a_1,\dots\rangle$ the iteration of $a$ under the action of $j$, i.e., $a_0=a$ and $a_{i+1}=j(a_i)$ for all $i\in\omega$. Then
    \begin{itemize}
      \item $j$ is \emph{weakly proper} if $j=j_U$;
      \item $j$ is \emph{proper} if it is weakly proper and if for every $X\in N$, $\langle X_i:i<\omega\rangle\in L(N)$ .
    \end{itemize}
  \end{defin}
	
	Properness was introduced because it implies iterability. In \cite{totally} and \cite{partially}, there are indicated some $L(N)$'s on which not all elementary embeddings are proper, sometimes even none.
	
	We call ${\cal E}(N)=\{j:j:N\prec N\}$ and we write $N<X$ if in $L(X,V_{\lambda+1})$ there exists a $\pi:V_{\lambda+1}\twoheadrightarrow N$.
	
	Now we can define the hypotheses above I0. They are a ``canonical'' sequence of $N$'s such that there exists $j:L(N)\prec L(N)$.
	
	\begin{defin}[Woodin, \cite{Woodin}]
   \label{E}
    Let $\lambda$ be a limit ordinal with cofinality $\omega$. The sequence
      \begin{equation*}
	\langle E^0_\alpha(V_{\lambda+1}):\alpha<\Upsilon_{V_{\lambda+1}}\rangle
      \end{equation*}
    is the maximum sequence such that the following hold:
      \begin{enumerate}
	\item $E^0_0(V_{\lambda+1})=L(V_{\lambda+1})\cap V_{\lambda+2}$;
	\item for $\alpha<\Upsilon_{V_{\lambda+1}}$ limit, $E^0_\alpha(V_{\lambda+1})=L(\bigcup_{\beta<\alpha} E^0_\beta(V_{\lambda+1}))\cap V_{\lambda+2}$; 
	\item for $\alpha<\Upsilon_{V_{\lambda+1}}$ limit,
	  \begin{itemize}
	    \item if $L(E^0_\alpha(V_{\lambda+1}))\vDash\cof(\Theta^{E^0_\alpha(V_{\lambda+1})})<\lambda$ then
	      \begin{equation*}
		E^0_{\alpha+1}(V_{\lambda+1})=L((E^0_\alpha(V_{\lambda+1}))^\lambda)\cap V_{\lambda+2};
	      \end{equation*}
	    \item if $L(E^0_\alpha(V_{\lambda+1}))\vDash\cof(\Theta^{E^0_\alpha(V_{\lambda+1})})>\lambda$ then
	      \begin{equation*}
		E^0_{\alpha+1}(V_{\lambda+1})=L({\cal E}(E^0_\alpha(V_{\lambda+1})))\cap V_{\lambda+2};
	      \end{equation*}
	  \end{itemize}
	\item for $\alpha=\beta+2<\Upsilon_{V_{\lambda+1}}$, there exists $X\subseteq V_{\lambda+1}$ such that $E^0_{\beta+1}(V_{\lambda+1})=L(X,V_{\lambda+1})\cap V_{\lambda+2}$ and $E^0_\beta(V_{\lambda+1})<X$, and
	  \begin{equation*}
	    E^0_{\beta+2}=L((X,V_{\lambda+1})^\sharp)\cap V_{\lambda+2}
	  \end{equation*}
	\item \label{less} $\forall\alpha<\Upsilon_{V_{\lambda+1}}\ \exists X\subseteq V_{\lambda+1}$ such that $E^0_\alpha(V_{\lambda+1})\subset L(X,V_{\lambda+1})$, $\exists j\colon L(X,V_{\lambda+1})\prec L(X,V_{\lambda+1})$ proper;			
	\item \label{HODZ} $\forall\alpha$ limit, $\alpha+1<\Upsilon_{V_{\lambda+1}}$ iff
	  \begin{multline*}
	    \text{if }L(E^0_\alpha(V_{\lambda+1}))\vDash\cof(\Theta^{E^0_\alpha(V_{\lambda+1})})>\lambda\\
	    \text{then }\exists Z\in E^0_\alpha(V_{\lambda+1})\  L(E^0_\alpha(V_{\lambda+1}))\vDash V=\HOD_{V_{\lambda+1}\cup\{Z\}}.
	  \end{multline*}
      \end{enumerate}
  \end{defin}

 For the rest of the paper we will use just the notation $E^0_\alpha$ instead of $E^0_\alpha(V_{\lambda+1})$ and $\Upsilon$ instead of $\Upsilon_{V_{\lambda+1}}$. It is not important for the purpose of this paper what the exact definition of $X^\sharp$ is, it is just some kind of description of the truth in $L(X,V_{\lambda+1})$.  
 
 What is the intuition behind this complex definition? The idea is to find a hierarchy of $L(X,V_{\lambda+1})$ equipped with elementary embeddings that is as much canonical as possible. The first step is to notice that the existence of a $j:L(V_{\lambda+1})\prec L(V_{\lambda+1})$ is equivalent to $j:(V_{\lambda+1},(V_{\lambda+1})^\sharp)\prec (V_{\lambda+1},(V_{\lambda+1})^\sharp)$. It is therefore natural to define the first step above I0 as the existence of 
\begin{equation*}
 j:L(V_{\lambda+1},(V_{\lambda+1})^\sharp)\prec L(V_{\lambda+1},(V_{\lambda+1})^\sharp), 
\end{equation*}
or $j:(V_{\lambda+1},(V_{\lambda+1})^{\sharp\sharp})\prec (V_{\lambda+1},(V_{\lambda+1})^{\sharp\sharp})$. Moreover, all the I0-embeddings are in $L(V_{\lambda+1},(V_{\lambda+1})^\sharp)$, so it is a step that really transcends I0.

The first idea is therefore to consider a hierarchy where every step is the sharp of the precedent one. But there can be a problem at the limit stage: it is the union of the previous stages, and it is possible that it is a model that cannot be described as $L(X,V_{\lambda+1})$, with $X\subseteq V_{\lambda+1}$. So, let $\alpha$ be limit. We want $L(E^0_{\alpha+1})$ to be some $L(X,V_{\lambda+1})$: Instead of adding $(E^0_\alpha)^\sharp$, we add something slightly smaller (depending on the cofinality of the relative $\Theta$). We have three possibilities:
\begin{itemize}
 \item There exists $Y\subseteq V_{\lambda+1}$ such that $L(E^0_\alpha)=L(Y,V_{\lambda+1})$. Then Lemma 28 and Theorem 31 in \cite{Woodin} prove that $L(E^0_{\alpha+1})=L(Y^\sharp,V_{\lambda+1})$, so this step is what we expect;
 \item There is no $Y\subseteq V_{\lambda+1}$ such that $L(E^0_\alpha)=L(Y,V_{\lambda+1})$, but if $E^0_\alpha$ is ``small enough'', then there exists $X\subseteq V_{\lambda+1}$ such that $L(E^0_{\alpha+1})=L(X,V_{\lambda+1})$, and for any $Y\in E^0_\alpha$, $Y^\sharp\in L(E^0_{\alpha+1})$;
 \item There is no $Y\subseteq V_{\lambda+1}$ such that $L(E^0_\alpha)=L(Y,V_{\lambda+1})$, and $E^0_\alpha$ is not ``small enough'', then the construction stops.
\end{itemize}

So the sequence can end for three reasons:
\begin{itemize}
 \item there are no more proper embeddings (i.e., there is no proper elementary embedding from the eventual $L(E^0_{\alpha+1})$);
 \item there are no more sharps (i.e., $E^0_{\alpha+1}$ cannot even be constructed);
 \item it is not possible to do the successor of the limit stage, as above.
\end{itemize}

Why is this sequence canonical? It is a consequence of Theorem 34 in \cite{Woodin}:

\begin{teo}
 For $X\subseteq V_{\lambda+1}$, if there exists a proper $j:L(X,V_{\lambda+1})\prec L(X,V_{\lambda+1})$, $\crt(j)<\lambda$, then $E^0_\alpha\in L(X,V_{\lambda+1})$ for all $\alpha$'s such that $\Theta^{E^0_\alpha}\leq \Theta^{L(X,V_{\lambda+1})}$, and $L(X,V_{\lambda+1})\vDash \exists k:L(E^0_\alpha)\prec L(E^0_\alpha)$.
\end{teo}

Its definition is also very nicely uniform. We only need these two theorems about this:

\begin{lem}
  \label{Condensation}
    Let $\beta<\Upsilon$, let $M$ be a model of \ZF{} such that $E^0_\beta\subseteq M$ and let $\bar{M}$ be $M$'s transitive collapse. If $M$ is an elementary substructure of $L(E^0_\eta)$ for some $\eta<\Upsilon$, then there exists $\beta\leq\gamma\leq\eta$ such that either $\bar{M}=L(E^0_\gamma)$ or else $\bar{M}=L_\zeta(E^0_\gamma)$ for some $\zeta$. Moreover, if $j:\bar{M}\prec L(E^0_\eta)$ is the inverse of the collapse, then $j(\gamma)=\eta$.
  \end{lem}
	
	\begin{lem}[\cite{Woodin}]
  \label{DefinablefromTheta}
    Suppose $\alpha<\Upsilon$ is a limit ordinal and $(\cof(\Theta^{E^0_\alpha}))^{L(E^0_\alpha)}>\lambda$. Then there exists $Z\in E^0_\alpha$ such that for each $Y\in E^0_\alpha$, $Y$ is $\Sigma_1$-definable in $L(E^0_\alpha)$ with parameters from $\{Z\}\cup\{V_{\lambda+1}\}\cup V_{\lambda+1}\cup\Theta^{E^0_\alpha}$. Moreover, if $L(E^0_\alpha)\vDash V=\HOD_{V_{\lambda+1}}$, then $Z=\emptyset$.
  \end{lem}

\section{The LD-algebras beyond I0}
 Our purpose now is to fix an $\alpha<\Upsilon$, and find a suitable definition for the application for embeddings $j:L(E^0_\alpha)\prec L(E^0_\alpha)$. We will not do it for all $\alpha$'s, but only for a good quantity of them in an initial segment of $\Upsilon$. Moreover, we will do this only for weakly proper embeddings, but thanks to Lemma \ref{LosN} this does not reduce generality.

 Let $\alpha<\Upsilon$ so that $L(E^0_\alpha)\vDash V=\HOD_{V_{\lambda+1}}$ and $(\cof(\Theta^{E^0_\alpha}))^{L(E^0_\alpha)}>\lambda$ (for example if $L(E^0_\alpha)=L(X,V_{\lambda+1})$, so if $\alpha$ is a successor, then $\Theta^{E^0_\alpha}$ is regular in $L(E^0_\alpha)$, for the same reason that $\Theta$ is regular in $L(\mathbb{R})$). Let $j,k:L(E^0_\alpha)\prec L(E^0_\alpha)$ be weakly proper. We define $j(k)$. The idea is still to cut $k$ in small pieces, transfer them through $j$ and reassemble them, but to preserve elementarity we need to choose the pieces in a smart way. The pieces will be elementary submodels of $L(E^0_\alpha)$, and will form a directed system that goes up to $L(E^0_\alpha)$.

Let $I_{j,k}=\{\alpha:j(\alpha)=\alpha,\ k(\alpha)=\alpha\}$. As $j,k$ are ultrapower embeddings, all the strong limit cardinals of cofinality $>\Theta^{L(E^0_\alpha)}$ are in $I_{j,k}$, therefore $I_{j,k}$ is a proper class. 

For an $s\in I_{j,k}^{<\omega}$ let
\begin{equation*}
 Z_s=H^{L(E^0_\alpha)}(s\cup V_{\lambda+1}\cup\{V_{\lambda+1}\}\cup\{E^0_\alpha\}\cup\Theta^{L(E^0_\alpha)}). 
\end{equation*}

\begin{claim}
 $\bigcup_{s\in I_{j,k}^{<\omega}}Z_s\prec L(E^0_\alpha)$ and its transitive collapse is $L(E^0_\alpha)$.
\end{claim} 
\begin{proof}
 Consider 
 \begin{equation*}
  Z=H^{L(E^0_\alpha)}(I_{j,k}\cup V_{\lambda+1}\cup\{V_{\lambda+1}\}\cup\{E^0_\alpha\}\cup\Theta^{L(E^0_\alpha)}). 
 \end{equation*}
Then $Z=\bigcup_{s\in I_{j,k}^{<\omega}}Z_s$, and it is an elementary submodel of $L(E^0_\alpha)$. By Lemma \ref{DefinablefromTheta} we have that $E^0_\alpha\subseteq Z$. By Lemma \ref{Condensation} the transitive collapse of $Z$ is $L(E^0_\alpha)$. 
\end{proof}

Now define $j(k)$ on every $Z_s$ as $j(k)\upharpoonright Z_s=j(k\upharpoonright Z_s)$. It is an elementary embedding from $Z_s$ to $Z_s$, as $s$, $V_{\lambda+1}$, $\{V_{\lambda+1}\}$, $\{E^0_\alpha\}$ are all fixed points of both $j$ and $k$. Now, the $Z_s$'s form a directed system with limit $Z$. Let $\bar{j(k)}$ be the corresponding induced limit. Then $j(k)$ is the embedding from $L(E^0_\alpha)$ to itself that is the composition of $\bar{j(k)}$ with the collapses of $Z$. We can suppose that $j(k)$ is weakly proper by Theorem \ref{LosN}.

Note that the construction does not depend crucially on $I_{j,k}$: let $p$ the collapse of $Z$ to $L(E^0_\alpha)$ and $x\in L(E^0_\alpha)$. Then $p^{-1}(x)\in Z_s$ is definable from $s\in (I_{j,k})^{<\omega}$, $a\in (V_{\lambda+1}\cup\{V_{\lambda+1}\}\cup\{E^0_\alpha\})^{<\omega}$ and $t\in(\Theta^{L(E^0_\alpha)})^{<\omega}$. As $V_{\lambda+1}$ and $\Theta^{L(E^0_\alpha)}$ are not collapsed, $x$ is definable from $p(s)$, $a$ and $t$. By elementarity $\bar{j(k)}(p^{-1}(x))$ is definable from $s$, $j(k)(a)$ and $j(k)(t)$, and therefore $j(k)(x)$ is definable from $p(s)$, $j(k)(a)$ and $j(k)(t)$. So if we use in the definition of $j(k)$ a proper subclass of $I_{j,k}$ instead of $I_{j,k}$ itself, the definition of $j(k)$ is the same.

We call ${\cal E}(E^0_\alpha)$ the ``set'' of weakly proper elementary embeddings from $L(E^0_\alpha)\prec L(E^0_\alpha)$. As the embeddings are classes the definition is not formally correct, but every embedding is coded by an ultrafilter, therefore we can code ${\cal E}(E^0_\alpha)$ as a set.

\begin{teo}
 $({\cal E}(E^0_\alpha),\cdot)$ is an LD-algebra, i.e., $j(k(h))=j(k)(j(h))$.
\end{teo} 
\begin{proof}
 Define all the applications using $I=I_{j,k}\cap I_{k,h}$. We identify the embeddings with the composition of the collapses of the corresponding $Z$. Pick $s\in I^{<\omega}$. Then 
 \begin{equation*}
  j(k(h))\upharpoonright Z_s=j(k(h)\upharpoonright Z_s)=j(k(h\upharpoonright Z_s)). 
 \end{equation*}
On the other hand 
\begin{equation*}
 j(k)(j(h))\upharpoonright Z_s=j(k)(j(h)\upharpoonright Z_s)=j(k)(j(h\upharpoonright Z_s)), 
\end{equation*}
because $s$ is also a fixed point of $j(k)$. Now let $t$ be such that $h\upharpoonright Z_s\in Z_t$ (so also $j(h\upharpoonright Z_s)\in Z_t$). Then 
\begin{gather*}
 j(k)(j(h\upharpoonright Z_s))=(j(k)\upharpoonright Z_t)(j(h\upharpoonright Z_s))=j(k\upharpoonright Z_t)(j(h\upharpoonright Z_s))=\\
 =j((k\upharpoonright Z_t)(h\upharpoonright Z_s))=j(k(h\upharpoonright Z_s)).
\end{gather*}

\end{proof}

\begin{rem}
 Let $\rho:{\cal E}(E^0_\alpha)\to {\cal E}_\lambda$ the intersection with $V_\lambda$. Then $\rho$ is a homomorphism.
\end{rem}
\begin{proof}
 This is immediate by the definition of application, as 
 \begin{equation*}
 \rho(j(k))=j(k)\upharpoonright V_\lambda=j(k\upharpoonright V_\lambda)=(j\upharpoonright V_\lambda)^+(k\upharpoonright V_\lambda)=\rho(j)\cdot\rho(k). 
 \end{equation*}
\end{proof}

\begin{prop}
Let ${\cal A}_j$ the subalgebra generated by $\{j\}$. Then ${\cal A}_j$ is free.
\end{prop}
\begin{proof}
Consider this diagram:
\newline
\begin{center}
\begin{tikzcd}
 T_1 \arrow{rd}{\pi_2} \arrow{r}{\pi_1} & {\cal A}_j \arrow{d}{\rho} \\
                                    & {\cal A}_{\rho(j)}
\end{tikzcd} 

\end{center}

\begin{tikzpicture}
 
\end{tikzpicture}

where $T_1$ is the set of words with one generator, $\pi_1$ is the surjective homomorphism to ${\cal A}_j$ that comes from the universality of $F_1$, and $\pi_2$ is the same for ${\cal A}_{\rho(j)}$. We claim that the diagram commutes, i.e., $\rho(\pi_1(t))=\pi_2(t)$ for any $t\in T_1$. But this is proved by induction on the complexity of $t$: if $t=x_1$, then $\rho(\pi_1(x_1))=\rho(j)=\pi_2(x_1)$. If $t=t_1\cdot t_2$, then 
\begin{equation*}
  \rho(\pi_1(t))=\rho(\pi_1(t_1)\cdot \pi_1(t_2))=\rho(\pi_1(t_1))\cdot\rho(\pi_1(t_2))=\pi_2(t_1)\cdot \pi_2(t_2)=\pi_2(t). 
\end{equation*}

To prove that ${\cal A}_j$ is free, we need to prove that if $\pi_1(t_1)=\pi_1(t_2)$, then $t_1\equiv_{LD}t_2$. But if $\pi_1(t_1)=\pi_1(t_2)$, then 
\begin{equation*}
  \pi_2(t_1)=\rho(\pi_1(t_1))=\rho(\pi_1(t_2))=\pi_2(t_2). 
\end{equation*}
As ${\cal A}_{\rho(j)}$ is free by Theorem \ref{free}, this implies that $t_1\equiv_{LD}t_2$.
\end{proof}

In other words, the freeness of ${\cal A}_{\rho(j)}$ implies that $\rho$ is actually an isomorphism between ${\cal A}_j$ and ${\cal A}_{\rho(j)}$. This means that going up the hierarchy of the $E^0_\alpha$'s actually does not have any effect on the algebra generated by one embedding.

Now we analyze the case with more generators. The case I0 does not add much information.

\begin{rem}
 Let $j,k:L(V_{\lambda+1})\prec L(V_{\lambda+1})$ be weakly proper. Then $j=k$ iff $j\upharpoonright V_\lambda=k\upharpoonright V_\lambda$. 
\end{rem}

Therefore $\rho:{\cal A}_{j,k}\to{\cal A}_{\rho(j),\rho(k)}$ is again an isomorphism, and there is no additional structure. 

Going up the hierarchy: if $j,k:L(X,V_{\lambda+1})\prec L(X,V_{\lambda+1})$ are proper and $j(X)=k(X)$, then $j=k$ iff $j\upharpoonright V_\lambda=k\upharpoonright V_\lambda$, so new structure appears only when $j(X)\neq k(X)$. The structure changes even more when we are considering non-proper embeddings:

Now let $\alpha$ be like in \cite{partially}:
\begin{teo}
 Suppose there exists $\xi<\Upsilon$ such that $L(E^0_\xi)\nvDash V=\HOD_{V_{\lambda+1}}$. Then there exists $\alpha<\xi$ such that
 \begin{itemize}
  \item $L(E^0_\alpha)\vDash V=\HOD_{V_{\lambda+1}}$;
	\item $\Theta^{L(E^0_\alpha)}$ is regular in $L(E^0_\alpha)$;
	\item there exist $j:L(E^0_\alpha)\prec L(E^0_\alpha)$ proper and $k:L(E^0_\alpha)\prec L(E^0_\alpha)$ weakly proper not proper such that $j\upharpoonright V_\lambda=k\upharpoonright V_\lambda$.
 \end{itemize}
\end{teo}

 Note that it satisfies our initial conditions. Then there are $j,k\in{\cal E}(E^0_\alpha)$ such that $j\neq k$ and $j\upharpoonright V_\lambda=k\upharpoonright V_\lambda$. Therefore $\rho$ is not an isomorphism, and the algebra generated by $j$ and $k$ is genuinely new.

\begin{rem}
 Suppose $j,k\in{\cal E}(E^0_\alpha)$ are weakly proper. Then $k$ is proper iff $j(k)$ is proper.
\end{rem}
\begin{proof}
 $k$ is proper iff 
 \begin{equation*}
  L(E^0_\alpha)\vDash\forall X\in E^0_\alpha\ \exists Y=\langle X_0,X_1,\dots\rangle\subseteq E^0_\alpha\ \forall n\in\omega\ X_{n+1}=(k\upharpoonright Z_{\emptyset})(X_n)
 \end{equation*}
 (still by Lemma \ref{DefinablefromTheta}). By elementarity the Remark follows. 
\end{proof}

\begin{cor}
 $k\notin{\cal E}_j$ and $j\notin{\cal E}_k$.
\end{cor}
\begin{proof}
 By the previous remark, all the embeddings in ${\cal E}_j$ are not proper and all the embeddings in ${\cal E}_k$ are proper.
\end{proof}

\section{Open Problems}

The embeddings $j$ and $k$ above have therefore some nice properties of independence. For example it is not possible that $j(k)=j$, as that would mean that $\rho(j)(\rho(k))=\rho(j)(\rho(j))=\rho(j)$, and this is impossible because ${\cal A}_{\rho(j)}$ is free. Moreover, it is not possible that $j(k)=k(j)$, as $j(k)$ is proper and $k(j)$ is not. The study on $V_\lambda$ gives even more results:

\begin{teo}[Laver-Steel Theorem \cite{Steel}]
 For any $j\in{\cal E}_\lambda$, there are no $j_1,\dots,j_n\in {\cal E}_\lambda$ so that $j=(\dots((j\cdot j_1)\cdot j_2)\dots\cdot j_n)$.
\end{teo}

Via $\rho$, this is true also in ${\cal A}_{j,k}$. 

Unfortunately, nothing is known about whether it is possible to have $j(k)\neq k(k)$, and similars.

\begin{Q}
 Are there $j$ proper, $k$ non proper such that the algebra generated by $j$ and $k$ is free?
\end{Q} 

The difficulty in achieving such a result is in the fact that the criterion for freeness of the many-generators algebra is distinct from the monogenic case:

\begin{teo}[Laver's Criterion \cite{Laver1}]
 Let $w_1,w_2\in T_2$. We define $w_1\leq_L w_2$ iff there are $u_1,\dots,u_n\in T_X$ so that $w_2=(\dots((w_1\cdot u_1)\cdot u_2)\dots\cdot u_n)$. Then a monogenic LD-algebra is free iff $\leq_L$ has no cycle.
\end{teo}

In our case, thanks to Laver-Steel Theorem, we do have this, but the criterion for the many-generators case is the following (Proposition 6.6 in Chapter 5 of \cite{Dehornoy2})\footnote{The author thanks the anonymous referee for having pointed out this criterion.}:

\begin{teo}[Dehornoy's Criterion]
 An LD-algebra $S$ with set of generators $X$ is free iff $\leq_L$ has no cycle and $S$ is quasi-free in $X$, i.e., no equality of the form $(\dots((((c_1\dots)\cdot c_r)\cdot x)a_1)\dots)\cdot a_p=(\dots((((c_1\dots)\cdot c_r)\cdot y)\cdot b_1)\dots)\cdot b_q$ holds.
\end{teo}

There is therefore a second case to check, that involves a disparate set of words (for example, the inequality $j(k)\neq k(j)$ is in this case). As words of different length can be equivalent under LD, there is no apparent order in them, so induction is difficult to implement. Results like those in \cite{Miller} could be needed to put some order first in such words, and exploit it to carry on some inductive proof. 

Another direction the research could take is forcing. Forcing is suspiciously absent in the analysis of the algebra of elementary embeddings, and yet it turned out to be profitable in the I0 case (see for example \cite{Dimonte} or \cite{ShiTrang}), thanks to a tool called `generic absoluteness'. New results that stems from \cite{Cramer} show that generic absolutenss could hold even in the $E^0_\alpha$ hierarchy, and therefore bring new results in the structure of proper and non-proper elementary embeddings.

\emph{Acknowledgments}. The paper was developed under the Italian program ``Rita Levi Montalcini 2013''. The author would also like to thank Sheila Miller, for having raised my interest on the subject of LD-algebras, and Liuzhen Wu and Gabriel Goldberg for the meaningful discussions.

\end{document}